\documentclass[a4paper,12pt,intlimits,oneside]{amsart}

\usepackage{enumerate}
\usepackage{mathrsfs}
\usepackage{amsfonts,amsmath}

\usepackage{latexsym,amssymb,amsthm}

\allowdisplaybreaks

\newcommand{\comment}[1]{}

\newcommand{\bC}{{\mathbb C}}
\newcommand{\bN}{{\mathbb N}}

\newcommand{\R}{{\mathbb R}}

\def\Bergman{{\mathcal A_\alpha^p(\mathbb C_+)}}

\def\H{{\mathcal H}}
\def\Hardy{{\mathcal H^p_{|\cdot|^\alpha}(\mathbb C_+)}}

\def\Hau{{\mathscr H_\varphi}}

\def\gsim{\raisebox{-1ex}{$~\stackrel{\textstyle >}{\sim}~$}}

\newcounter{rea}
\newcounter{rek}
\newcounter{res}
\newcommand{\ackname}{Acknowledgements}
\makeatletter
\if@titlepage
\newenvironment{acknowledgement}{%
	\titlepage
	\null\vfil
	\@beginparpenalty\@lowpenalty
	\begin{center}%
		\bfseries \ackname
		\@endparpenalty\@M
\end{center}}%
{\par\vfil\null\endtitlepage}
\else
\newenvironment{acknowledgement}{%
	\if@twocolumn
	\section*{\ackname}%
	\else
	\begin{center}%
		{\bfseries \ackname\vspace{0em}\vspace{\z@}}%
	\end{center}%
	\quotation
	\fi}
\fi
\makeatother

\begin{document}
\title[]{Hausdorff operators on weighted Bergman and Hardy spaces}         
\author{Ha Duy Hung}    
\address{High School for Gifted Students,
	Hanoi National University of Education, 136 Xuan Thuy, Hanoi, Vietnam} 
\email{{\tt hunghd@hnue.edu.vn}}
\author{Luong Dang Ky$^*$}
\address{Department of Education, Quy Nhon University, 
170 An Duong Vuong, Quy Nhon, Binh Dinh, Vietnam} 
\email{{\tt luongdangky@qnu.edu.vn}}
\keywords{Hausdorff operator, weighted Hardy space, weighted Bergman space, holomorphic function}
\subjclass[2010]{47B38, 42B30, 46E15}
\thanks{$^*$Corresponding author}

\begin{abstract} 
	Let $1\leq p<\infty$, $\alpha>-1$, and let $\varphi$ be a measurable function on $(0,\infty)$. The main purpose of this paper is to study the Hausdorff operator 
	\[ \mathscr H_\varphi f(z)=\int_0^\infty f\left(\frac{z}{t}\right) \frac{\varphi(t)}{t} dt, \quad z\in \mathbb C^+, \] 
	on the weighted Bergman space $\mathcal A^p_\alpha(\mathbb C_+)$ and on the power weighted Hardy space $\mathcal H^p_{|\cdot|^\alpha}(\mathbb{C_+})$ of the upper half-plane. Some applications to the real version of $\mathscr H_\varphi$ are also given.  
\end{abstract}

\maketitle
\newtheorem{theorem}{Theorem}[section]
\newtheorem{lemma}{Lemma}[section]
\newtheorem{proposition}{Proposition}[section]
\newtheorem{remark}{Remark}[section]
\newtheorem{corollary}{Corollary}[section]
\newtheorem{definition}{Definition}[section]
\newtheorem{example}{Example}[section]
\numberwithin{equation}{section}
\newtheorem{Theorem}{Theorem}[section]
\newtheorem{Lemma}{Lemma}[section]
\newtheorem{Proposition}{Proposition}[section]
\newtheorem{Remark}{Remark}[section]
\newtheorem{Corollary}{Corollary}[section]
\newtheorem{Definition}{Definition}[section]
\newtheorem{Example}{Example}[section]
\newtheorem*{theorema}{Theorem A}

\section{Introduction and main results} 

Let $\varphi$ be a measurable function on $(0,\infty)$ and $\bC_+:=\{x+iy\in\bC: y>0\}$ be the upper half-plane in the complex plane. Following \cite{HKQ20}, the {\it Hausdorff operator} $\mathscr H_\varphi$ associated with the kernel $\varphi$ is defined for suitable holomorphic functions $f$ on $\bC_+$ by
$$\mathscr H_\varphi f(z)=\int_{0}^{\infty} f\left(\frac{z}{t}\right) \frac{\varphi(t)}{t} dt,\quad z\in\bC_+.$$
The real version of $\mathscr H_\varphi$ is defined for suitable measurable functions $f^*$ on $\R$ by
$$\H_\varphi (f^*)(x)=\int_{0}^{\infty} f^*\left(\frac{x}{t}\right) \frac{\varphi(t)}{t} dt,\quad x\in \R.$$

Hausdorff operators (Hausdorff summability methods) appeared long ago
aiming to solve certain classical problems in analysis. The theory of Hausdorff summability started with the paper of Hausdorff \cite{Ha} in 1921. The modern theory of Hausdorff operators began with the work of Siskakis and Galanopoulos \cite{GSi, Si87, Si90, Si96} in the complex analysis setting and with the work of Georgakis \cite{Ge} and Liflyand-M\'oricz \cite{LM00} in the Fourier transform setting on the real line. The Hausdorff operator is an interesting operator in harmonic analysis. There are many classical operators in analysis which are special cases of the Hausdorff operator if one chooses suitable kernel functions $\varphi$, such as the classical Hardy operator, its adjoint operator, the Ces\`aro type operators, the Erd\'elyi-Kober fractional integral operators, the de La Vall\'ee-Poussin type operators, the Picar and Bessel operators, the Stieltjes type operators, etc.  Also, the Riemann–Liouville fractional integral operators and the Weyl fractional integral operators can be derived from the Hausdorff operators. See the survey article \cite{Li13} and the references therein. In recent years, there is an increasing interest in the study of boundedness of the Hausdorff operators on the Lebesgue and real Hardy spaces (see, for example, \cite{HKQ18, LM19, LM00, LMPS, MO}), on the power weighted Hardy spaces (see, for example, \cite{Ky, RF16}), and on some spaces of holomorphic functions (see, for example, \cite{AS, BBMM, BG, Bo, GP, GSt, HKQ20, Mi21, St}).

 Let $0<p<\infty$ and $\alpha>0$, the {\it weighted Bergman space} $\Bergman$ is defined as the set of all holomorphic functions $f$ on $\mathbb C_+$ such that
$$\|f\|_{\Bergman}:=\left(\int_{\bC_+} |f(z)|^p dA_\alpha(z)\right)^{1/p}<\infty,$$
where $dA_\alpha(z)= y^{\alpha-1} dx dy$ for all $z=x+iy\in\bC_+$. A holomorphic function $f$ on $\mathbb C_+$ is said to belong to the {\it Dirichlet space} $\mathcal D(\bC_+)$ if its derivative $f'$ is in the Bergman space $\mathcal A^2_1(\bC_+)$. Upon identifying functions
that differ by a constant, the space $\mathcal D(\bC_+)$ becomes a Hilbert space with norm
$$\|f\|_{\mathcal D(\bC_+)}:= \|f'\|_{\mathcal A^2_1(\bC_+)}=\left( \int_{\bC_+} |f'(z)|^2 dA_1(z)\right)^{1/2}<\infty.$$ 

Let $0< p\leq \infty$, $\alpha>-1$, and let $\varphi$ be a measurable function on $(0,\infty)$ which may or may not satisfy
\begin{equation}\label{main inequality}
	\int_0^\infty t^{\frac{1+\alpha}{p}-1} |\varphi(t)| dt<\infty.
\end{equation}

{\it Our first main result} reads as follows.

\begin{theorem}\label{the main theorem 1} 
	Let $1\leq p<\infty$, $\alpha>0$, and let $\varphi$ be a measurable function on $(0,\infty)$. Then:
	\begin{enumerate}[\rm (i)]
		\item If (\ref{main inequality}) holds, then $\Hau$ is bounded on $\Bergman$. Moreover,
		$$\left|\int_0^\infty t^{\frac{1+\alpha}{p}-1} \varphi(t) dt\right|\leq \|\Hau\|_{\Bergman \to \Bergman}\leq \int_0^\infty t^{\frac{1+\alpha}{p}-1} |\varphi(t)| dt.$$
		
		\item Conversely, if $\varphi\geq 0$ and $\Hau$ is bounded on $\Bergman$, then (\ref{main inequality}) holds.
	\end{enumerate}
\end{theorem}

\begin{corollary}\label{07:43, 19/11/2021}
	Let $1\leq p<\infty$, $\alpha>0$, and let $\varphi$ be a nonnegative  measurable function on $(0,\infty)$. Then, $\Hau$ is bounded on $\Bergman$ if and only if (\ref{main inequality}) holds, moreover,
	$$\|\Hau\|_{\Bergman \to \Bergman}= \int_0^\infty t^{\frac{1+\alpha}{p}-1} \varphi(t) dt.$$
\end{corollary}

\begin{remark}
	Our first main result generalizes some recent results in \cite{BBMM, St}.
\end{remark}

As a consequence of Theorem \ref{the main theorem 1}, we obtain the following result.

\begin{theorem}\label{16:22, 26/4/2025}
	Let  $\varphi$ be a locally integrable function  on $(0,\infty)$ such that $\int_0^\infty \frac{|\log t|}{t} |\varphi(t)| dt<\infty$. Then, $\Hau$ is bounded on $\mathcal D(\bC_+)$, moreover,
	$$\|\Hau\|_{\mathcal D(\bC_+) \to \mathcal D(\bC_+)}\leq \int_0^\infty \frac{|\varphi(t)|}{t} dt.$$
\end{theorem}

Given $1\leq q<\infty$, a nonnegative locally integrable function $w:\R\to [0,\infty)$ belongs to the {\it Muckenhoupt class} $A_q(\R)$, say $w\in A_q(\R)$, if there exists a constant $C>0$ so that
$$\frac{1}{|I|}\int_I w(x)dx \left(\frac{1}{|I|} \int_I [w(x)]^{-\frac{1}{q-1}} dx\right)^{q-1}\leq C,\quad \mbox{if $1<q<\infty$},$$
and
$$\frac{1}{|I|}\int_I w(x)dx\leq C \mathop{\mbox{ess-inf}}\limits_{x\in I} w(x),\quad \mbox{if $q=1$},$$
for all intervals $I\subset \R$. We say that $w\in A_\infty(\R)$ if $w\in A_q(\R)$ for some $1\leq q<\infty$.

\begin{remark}\label{10:11, 18/11/2021}
	Let $1<p<\infty$ and $-1<\alpha< p-1$. Then, it is well-known that the power weight function $|x|^\alpha$ belongs to the  Muckenhoupt class $A_p(\R)$, and thus the Hilbert transform $H$, which is defined by
	\begin{equation}
		H f(x)= \frac{1}{\pi}\; {\rm p.v.}\int_{-\infty}^{\infty}\frac{f(x-y)}{y}dy, \quad \mbox{a.e. $x\in\R$},
	\end{equation}
	is bounded on $L^p_{|\cdot|^\alpha}(\R)$. See the book of Garc\'ia-Cuerva and Rubio de Francia \cite[CHAPTER IV]{GR}. 
\end{remark}

Let  $0<p\leq \infty$ and $\alpha>-1$. Following Garc\'ia-Cuerva \cite{Ga}, we define the {\it power weighted Hardy space} $\H_{|\cdot|^\alpha}^p(\bC_+)$ as the set of all holomorphic functions $f$ on $\mathbb C_+$ such that
$$\|f\|_{\H_{|\cdot|^\alpha}^p(\mathbb C_+)}:= \sup_{y>0} \left(\int_{-\infty}^{\infty} |f(x+iy)|^p |x|^\alpha dx\right)^{1/p}<\infty$$
if $0<p<\infty$, and if $p=\infty$, then
$$\|f\|_{\H_{|\cdot|^\alpha}^\infty(\mathbb C_+)}:= \sup_{z\in\mathbb C_+} |f(z)|<\infty.$$

{\it Our second main result} reads as follows.

\begin{theorem}\label{the main theorem 2} 	
	Let $1\leq p\leq\infty$, $\alpha>-1$, and let $\varphi$ be a measurable function on $(0,\infty)$. Then:
	\begin{enumerate}[\rm (i)]
		\item If (\ref{main inequality}) holds, then $\Hau$ is bounded on $\Hardy$. Moreover,
		$$\left|\int_0^\infty t^{\frac{1+\alpha}{p}-1} \varphi(t) dt\right|\leq \|\Hau\|_{\Hardy \to \Hardy}\leq \int_0^\infty t^{\frac{1+\alpha}{p}-1} |\varphi(t)| dt.$$
		
		\item Conversely, if $\varphi\geq 0$ and $\Hau$ is bounded on $\Hardy$, then (\ref{main inequality}) holds.
	\end{enumerate}
\end{theorem}

\begin{corollary}	
	Let $1\leq p\leq\infty$, $\alpha>-1$, and let $\varphi$ be a nonnegative  measurable function on $(0,\infty)$. Then, $\Hau$ is bounded on $\Hardy$ if and only if (\ref{main inequality}) holds, moreover,
	$$\|\Hau\|_{\Hardy \to \Hardy}= \int_0^\infty t^{\frac{1+\alpha}{p}-1} \varphi(t) dt.$$
\end{corollary}

\begin{remark}
	Our second main result generalizes some recent results in \cite{AS, BBMM, HKQ20}.
\end{remark}

It is well-known (see \cite[THEOREM II.1.1]{Ga} and \cite[Theorem 5.3]{Garn}) that if $f\in \Hardy$, then $f$ has a {\it boundary value function} $f^*\in L^p_{|\cdot|^\alpha}(\R)$ defined by
\begin{equation}\label{22:32, 20/11/2021}
	f^*(x)=\lim_{y\to 0} f(x+iy),\quad \mbox{a.e. $x\in\R$}.
\end{equation}
If $f^*\in L^p_{|\cdot|^\alpha}(\R)$ is the boundary function of $f\in \Hardy$, the question arises whether the transformed function $\H_{\varphi}(f^*)$ is also the
boundary function of a function in $\Hardy$? As a consequence of Theorem \ref{the main theorem 2}, we obtain a positive answer to this question.

\begin{theorem}\label{16:37, 26/4/2025}
	Let $1\leq p\leq \infty$, $\alpha>-1$, and let $\varphi$ be a measurable function on $(0,\infty)$ such that (\ref{main inequality}) holds. Then, for every $f\in \Hardy$,
	$$(\Hau f)^*= \mathcal H_\varphi(f^*).$$
\end{theorem}

The following result is another consequence of Theorem \ref{the main theorem 2}.

\begin{theorem}\label{16:42, 26/4/2025}
	Let $1<p<\infty$, $\alpha>-1$, and let $\varphi$ be a measurable function on $(0,\infty)$ such that (\ref{main inequality}) holds. Then:
	\begin{enumerate}[\rm (i)]
		\item $\H_\varphi$ is bounded on $L^p_{|\cdot|^\alpha}(\R)$, moreover,
		$$\left|\int_{0}^{\infty} t^{\frac{1+\alpha}{p}-1} \varphi(t) dt\right|\leq \|\H_\varphi\|_{L^p_{|\cdot|^\alpha}(\R)\to L^p_{|\cdot|^\alpha}(\R)}\leq \int_{0}^{\infty} t^{\frac{1+\alpha}{p}-1} |\varphi(t)| dt.$$
		
		\item If $-1<\alpha<p-1$, then $\H_\varphi$ commutes with the Hilbert transform $H$ on $L^p_{|\cdot|^\alpha}(\R)$.
	\end{enumerate} 
\end{theorem}

\begin{corollary}
	Let $1<p<\infty$, $\alpha>-1$, and let $\varphi$ be a nonnegative measurable function on $(0,\infty)$ such that (\ref{main inequality}) holds. Then, $\H_\varphi$ is bounded on $L^p_{|\cdot|^\alpha}(\R)$, moreover,
	$$\|\H_\varphi\|_{L^p_{|\cdot|^\alpha}(\R)\to L^p_{|\cdot|^\alpha}(\R)}= \int_{0}^{\infty} t^{\frac{1+\alpha}{p}-1} \varphi(t) dt.$$
\end{corollary}

\begin{remark}
	\begin{enumerate}[\rm (i)]
		\item Theorem \ref{16:37, 26/4/2025} generalizes some recent results in \cite{AS, HKQ20}.
		
		\item Theorem \ref{16:42, 26/4/2025} generalizes some recent results in \cite{LMPS, MO}.
	\end{enumerate}
\end{remark}

The organization of the paper is as follows. In Section 2, we present some results concerning the test functions $\Phi_\varepsilon(z)=\frac{1}{(z+ \varepsilon i)^{\frac{1+\alpha}{p}+\varepsilon}}$ with $\varepsilon>0$. In Section 3, we give the proofs of Theorems \ref{the main theorem 1} and \ref{16:22, 26/4/2025}. Section 4 is devoted to the proof of Theorem \ref{the main theorem 2}. Finally, Section 5 provides some applications of Theorem  \ref{the main theorem 2} in studying the Hausdorff operators $\mathcal H_\varphi$ on the real function spaces; the proofs of Theorems \ref{16:37, 26/4/2025} and \ref{16:42, 26/4/2025} are also given in this section.

Throughout the whole article, we use the symbol $A \lesssim B$ (or $B\gtrsim A$) means that $A\leq C B$ where $C$ is a positive constant which is independent of the main parameters, but it may vary from line to line. If $A \lesssim B$ and $B\lesssim A$, then we  write $A\sim B$.  For any $E\subset \R$, we denote by $\chi_E$ its characteristic function. For any $z\in\bC$, we denote by $\Re(z)$, $\Im(z)$ and $\arg(z)$ its real part, imaginary part and argument, respectively.

\section{Test functions}

Let $1\leq p<\infty$ and $\alpha>-1$. For any $\varepsilon>0$, following \cite{HKQ20, St},  we consider the test functions which are defined as follows 
\begin{equation}\label{key test function}
\Phi_\varepsilon(z)=\frac{1}{(z+ \varepsilon i)^{\frac{1+\alpha}{p}+\varepsilon}}
\end{equation}
and
\begin{equation}\label{15:56, 06/7/2021}
\psi_\varepsilon(z)= \frac{\overline{z+ \varepsilon i}}{|z+ \varepsilon i|}
\end{equation}
for every $z\in \bC_+$. For any $-\pi\leq a<b\leq \pi$, set
\begin{equation}\label{15:00, 05/7/2021}
A_{(a,b)}:=\left\{z\in\mathbb C_+: a < \arg(z) < b\right\}
\end{equation}
and 
\begin{equation}\label{15:02, 05/7/2021}
S_{(a,b)}:=\left\{z\in A_{(a,b)}: |z|\geq 1\right\}.
\end{equation}

\begin{remark}\label{10:09, 08/7/2021}
	Let $1\leq p<\infty$ and $\alpha>-1$. Then:
	\begin{enumerate}[\rm (i)]
		\item For any $0<\varepsilon<\frac{\pi}{2}$, we have
		$$\|\Phi_\varepsilon\|_{\Bergman} \sim \varepsilon^{-(\frac{1}{p}+\varepsilon)}$$
		and
		$$\|\Phi_\varepsilon\|_{\Hardy} \sim \varepsilon^{-(\frac{1}{p}+\varepsilon)},$$
		where the positive constants $C$ depend only on $p$ and $\alpha$.
		
		\item For any $\varepsilon>0$, we have $$\Re[\Phi_\varepsilon(z)]=\cos\left[\left(\frac{1+\alpha}{p}+\varepsilon\right)\arg(\psi_\varepsilon(z))\right] |\Phi_\varepsilon(z)|$$
		and $$\Im[\Phi_\varepsilon(z)]=\sin\left[\left(\frac{1+\alpha}{p}+\varepsilon\right)\arg(\psi_\varepsilon(z))\right] |\Phi_\varepsilon(z)|.$$
		Furthermore, if $\varepsilon_0\in (0,\frac{\pi}{2})$ and $z\in A_{(\frac{\pi}{2}-\varepsilon_0, \frac{\pi}{2})}$, then
		$$- \frac{\pi}{2}<\arg(\psi_\varepsilon(z))<\varepsilon_0 - \frac{\pi}{2}.$$
	\end{enumerate}	
\end{remark}

In order to prove Theorems \ref{the main theorem 1} and \ref{the main theorem 2}, we need the following key lemma.

\begin{lemma}\label{14:44, 08/7/2021}
	Let $1\leq p<\infty$ and $\alpha>-1$.
	\begin{enumerate}[\rm (i)]
		\item If $\frac{1+\alpha}{p}\notin \{2n: n\in\bN\}$, then there exist two constants $\varepsilon(p,\alpha)\in (0,\frac{\pi}{2})$ and $C(p,\alpha)>0$ such that 
		$$\begin{cases}
			\mbox{$\Im(\Phi_\varepsilon)$ has constant sign on $A_{(\frac{\pi}{2}-\varepsilon(p,\alpha), \frac{\pi}{2})}$ for all $0<\varepsilon<\varepsilon(p,\alpha)$},\\
			\mbox{$|\Im\left[\Phi_\varepsilon(z)\right]|\geq C(p,\alpha) \left|\Phi_\varepsilon(z)\right|$ for all $z\in A_{(\frac{\pi}{2}-\varepsilon(p,\alpha), \frac{\pi}{2})}$}.
		\end{cases}$$
		
		\item If $\frac{1+\alpha}{p}\in \{2n: n\in\bN\}$, then there exist two constants $\varepsilon(p,\alpha)\in (0,\frac{\pi}{2})$ and $C(p,\alpha)>0$ such that 
		$$\begin{cases}
		\mbox{$\Re(\Phi_\varepsilon)$ has constant sign on $A_{(\frac{\pi}{2}-\varepsilon(p,\alpha), \frac{\pi}{2})}$ for all $0<\varepsilon<\varepsilon(p,\alpha)$},\\
		\mbox{$|\Re\left[\Phi_\varepsilon(z)\right]|\geq C(p,\alpha) \left|\Phi_\varepsilon(z)\right|$ for all $z\in A_{(\frac{\pi}{2}-\varepsilon(p,\alpha), \frac{\pi}{2})}$}.
		\end{cases}$$
	\end{enumerate}
\end{lemma}

\begin{proof}

	(i) From $0< \frac{1+\alpha}{p}\notin \{2n: n\in\bN\}$, it follows that there exists $k=k(p,\alpha)\in \bN\cup\{0\}$ such that $\frac{1+\alpha}{p} \in (4k, 4k+2)\cup (4k+2, 4k+4)$.
	
	{\it Case 1:} $\frac{1+\alpha}{p} \in (4k, 4k+2)$. Then, there is a constant $\varepsilon(p,\alpha)\in (0,\pi/2)$ small enough for which
	$$2k\pi < \frac{1+\alpha}{p} \left(\frac{\pi}{2}-\varepsilon(p,\alpha)\right) \leq	\left(\frac{1+\alpha}{p}+\varepsilon(p,\alpha)\right)\frac{\pi}{2} <2k\pi +\pi.$$
	Hence,
	$$\begin{aligned}
		2k\pi <  \frac{1+\alpha}{p} \left(\frac{\pi}{2}-\varepsilon(p,\alpha)\right) &\leq\left(\frac{1+\alpha}{p}+\varepsilon\right)\theta\\
		&\leq \left(\frac{1+\alpha}{p}+\varepsilon(p,\alpha)\right)\frac{\pi}{2} <2k\pi +\pi
	\end{aligned}$$
	for all $0<\varepsilon<\varepsilon(p,\alpha)$ and all $\frac{\pi}{2}-\varepsilon(p,\alpha)<\theta<\frac{\pi}{2}$. This, together with Remark \ref{10:09, 08/7/2021}(ii), allows us to conclude that 
	$$\Im \left[\Phi_\varepsilon(z)\right]=\sin\left[\left(\frac{1+\alpha}{p}+\varepsilon\right)\arg(\psi_\varepsilon(z))\right]|\Phi_\varepsilon(z)|<0$$
	and
	$$|\Im \left[\Phi_\varepsilon(z)\right]|\geq C(p,\alpha) |\Phi_\varepsilon(z)|$$
	for all $0<\varepsilon<\varepsilon(p,\alpha)$ and all $z\in A_{(\frac{\pi}{2}-\varepsilon(p,\alpha), \frac{\pi}{2})}$, where 
	$$C(p,\alpha):= \min_{\frac{1+\alpha}{p} \left(\frac{\pi}{2}-\varepsilon(p,\alpha)\right)\leq \zeta\leq \left(\frac{1+\alpha}{p}+\varepsilon(p,\alpha)\right)\frac{\pi}{2}}\sin \zeta>0.$$
	
	{\it Case 2:} $\frac{1+\alpha}{p} \in (4k+2, 4k+4)$. Then, there is a constant $\varepsilon(p,\alpha)\in (0,\pi/2)$ small enough for which
	$$2k\pi +\pi< \frac{1+\alpha}{p} \left(\frac{\pi}{2}-\varepsilon(p,\alpha)\right)\leq \left(\frac{1+\alpha}{p}+\varepsilon(p,\alpha)\right)\frac{\pi}{2} <2k\pi + 2\pi.$$
	Hence,
	$$\begin{aligned}
	2k\pi +\pi <  \frac{1+\alpha}{p} \left(\frac{\pi}{2}-\varepsilon(p,\alpha)\right) &\leq\left(\frac{1+\alpha}{p}+\varepsilon\right)\theta\\
	&\leq \left(\frac{1+\alpha}{p}+\varepsilon(p,\alpha)\right)\frac{\pi}{2} <2k\pi + 2\pi
	\end{aligned}$$
	for all $0<\varepsilon<\varepsilon(p,\alpha)$ and all $\frac{\pi}{2}-\varepsilon(p,\alpha)<\theta<\frac{\pi}{2}$.  This, together with Remark \ref{10:09, 08/7/2021}(ii), allows us to conclude that 
	$$\Im \left[\Phi_\varepsilon(z)\right]=\sin\left[\left(\frac{1+\alpha}{p}+\varepsilon\right)\arg(\psi_\varepsilon(z))\right] |\Phi_\varepsilon(z)|>0$$
	and
	$$|\Im \left[\Phi_\varepsilon(z)\right]|\geq C(p,\alpha) |\Phi_\varepsilon(z)|$$
	for all $0<\varepsilon<\varepsilon(p,\alpha)$ and all $z\in A_{(\frac{\pi}{2}-\varepsilon(p,\alpha), \frac{\pi}{2})}$,	where 
	$$C(p,\alpha):= \min_{\frac{1+\alpha}{p} \left(\frac{\pi}{2}-\varepsilon(p,\alpha)\right)\leq \zeta\leq \left(\frac{1+\alpha}{p}+\varepsilon(p,\alpha)\right)\frac{\pi}{2}}(-\sin \zeta)>0.$$
	
	(ii) Let us consider the following two cases.
	
	{\it Case 1:} There is a constant $k=k(p,\alpha)\in \bN$ such that $\frac{1+\alpha}{p}=4k$. Set $\varepsilon(p,\alpha):=\frac{p\pi}{3(1+\alpha)}=\frac{\pi}{12k}\in (0,\frac{2}{3}) \subset (0,\frac{\pi}{2})$, then
	$$2k\pi - \frac{\pi}{3} \leq\left(\frac{1+\alpha}{p}+\varepsilon\right)\theta\leq 2k\pi + \frac{\pi}{3}$$
	for all $0<\varepsilon<\varepsilon(p,\alpha)$ and all $\frac{\pi}{2}-\varepsilon(p,\alpha)<\theta< \frac{\pi}{2}$. This, together with Remark \ref{10:09, 08/7/2021}(ii), allows us to conclude that 	
	$$\Re\left[\Phi_\varepsilon(z)\right]=\cos\left[\left(\frac{1+\alpha}{p}+\varepsilon\right)\arg(\psi_\varepsilon(z))\right] |\Phi_\varepsilon(z)|\geq \frac{1}{2}|\Phi_\varepsilon(z)|>0,$$ 
	and thus
	$$\left|\Re\left[\Phi_\varepsilon(z)\right]\right|\geq \frac{1}{2}|\Phi_\varepsilon(z)|$$ 
	for all $0<\varepsilon<\varepsilon(p,\alpha)$ and all $z\in A_{(\frac{\pi}{2}-\varepsilon(p,\alpha), \frac{\pi}{2})}$.
	
	{\it Case 2:} There is a constant $k=k(p,\alpha)\in \bN\cup\{0\}$ such that $\frac{1+\alpha}{p}=4k +2$. Set $\varepsilon(p,\alpha):=\frac{p\pi}{3(1+\alpha)}=\frac{\pi}{3(4k+2)}\in (0,\frac{2}{3}) \subset (0,\frac{\pi}{2})$, then
	$$2k\pi  +\frac{2\pi}{3} \leq\left(\frac{1+\alpha}{p}+\varepsilon\right)\theta\leq 2k\pi + \frac{4\pi}{3}$$
	for all $0<\varepsilon<\varepsilon(p,\alpha)$ and all $\frac{\pi}{2}-\varepsilon(p,\alpha)<\theta< \frac{\pi}{2}$. This, together with Remark \ref{10:09, 08/7/2021}(ii), allows us to conclude that 	
	$$\Re\left[\Phi_\varepsilon(z)\right]=\cos\left[\left(\frac{1+\alpha}{p}+\varepsilon\right)\arg(\psi_\varepsilon(z))\right] |\Phi_\varepsilon(z)|\leq -\frac{1}{2}|\Phi_\varepsilon(z)|<0,$$
	and thus
	$$\left|\Re\left[\Phi_\varepsilon(z)\right]\right|\geq \frac{1}{2}|\Phi_\varepsilon(z)|$$ 
	for all $0<\varepsilon<\varepsilon(p,\alpha)$ and all $z\in A_{(\frac{\pi}{2}-\varepsilon(p,\alpha), \frac{\pi}{2})}$.

\end{proof}

\section{Hausdorff operators $\Hau$ acting on $\mathcal A^p_\alpha(\bC_+)$ and $\mathcal D(\bC_+)$}

The main purpose of this section is to give the proofs of Theorems \ref{the main theorem 1} and \ref{16:22, 26/4/2025}. Recall (see, for example, \cite[Lemma 1.1]{Garn}) the classical Schwarz lemma for the unit disc $\mathbb D=\{z\in \bC : |z| < 1\}$ in the complex plane.

\begin{lemma}\label{the Schwarz lemma}
	Let $f: \mathbb D\to\bC$  be a holomorphic function such that $f(0)=0$ and $|f(z)|\leq 1$ for every $z\in \mathbb D$. Then, $|f'(0)|\leq 1$ and $|f(z)|\leq |z|$ for every $z\in \mathbb D$.
\end{lemma}

In what follows, denote by
\begin{equation}\label{07:18, 19/11/2021}
	\widetilde{\varphi}(t):=\frac{\varphi(t)}{t},\quad \forall \, t\in (0,\infty),
\end{equation}
and, for any $0<\delta<1$,
\begin{equation}\label{14:48, 21/11/2021}
	\varphi_\delta(t):= \varphi(t) \chi_{(\delta,1/\delta)}(t),\quad \forall \, t\in (0,\infty).
\end{equation}


\begin{lemma}\label{10:57, 03/7/2021}
	Let $1\leq p<\infty$, $\alpha>0$, and let $\varphi$ be such that (\ref{main inequality}) holds. Then:
	\begin{enumerate}[\rm (i)]
		\item $\Hau f$ is well-defined and  holomorphic on $\mathbb C_+$ for all $f\in\Bergman$. Moreover, 
		$$(\Hau f)'= \mathscr H_{\widetilde{\varphi}} (f').$$

		\item $\Hau$ is bounded on $\Bergman$, moreover,
		$$\left|\int_0^\infty t^{\frac{1+\alpha}{p}-1}\varphi(t)dt\right|\leq \|\Hau\|_{\Bergman\to\Bergman}\leq \int_0^\infty t^{\frac{1+\alpha}{p}-1}|\varphi(t)| dt.$$
	\end{enumerate}
\end{lemma}

\begin{proof}
	
	(i) It suffices to consider the case $f\ne 0$. Thanks to \cite[Proposition 1.3]{BBGNPR}, we have
	\begin{equation}\label{BBGNPR, Proposition 1.3}
	|f(z)|\leq C(p,\alpha) y^{-\frac{1+\alpha}{p}} \|f\|_{\Bergman},\quad\forall\, z=x+iy\in\mathbb C_+, 
	\end{equation}
	where the constant $C(p,\alpha)>0$ depends only on $p$ and $\alpha$. Therefore,
	$$\begin{aligned}
		|\Hau f(z)| &\leq \int_0^\infty |f(z/t)| \frac{|\varphi(t)|}{t} dt\\
		&\leq C(p,\alpha) y^{-\frac{1+\alpha}{p}} \|f\|_{\Bergman} \int_0^\infty t^{\frac{1+\alpha}{p} -1} |\varphi(t)| dt<\infty
	\end{aligned}$$
	for every $z=x+iy\in\bC_+$. This proves that $\Hau f$ is well-defined  on $\mathbb C_+$.

	 On the other hand, for every $z_0=x_0+i y_0\in \bC_+$, we define the holomorphic function $g: \mathbb D\to \bC$ as follows
	$$g(z)=\frac{f(z_0+ \frac{y_0}{2}z)-f(z_0)}{(1+2^{\frac{1+\alpha}{p}})C(p,\alpha) y_0^{-\frac{1+\alpha}{p}}\|f\|_{\Bergman}}.$$
	Then, it follows from (\ref{BBGNPR, Proposition 1.3}) and Lemma \ref{the Schwarz lemma} that
	$$\begin{aligned}
		|f'(z_0)|&=\frac{2}{y_0}(1+2^{\frac{1+\alpha}{p}})C(p,\alpha) y_0^{-\frac{1+\alpha}{p}}\|f\|_{\Bergman}|g'(0)|\\
		&\leq C_1(p,\alpha)  y_0^{-\frac{1+\alpha}{p}-1}\|f\|_{\Bergman},
	\end{aligned}$$
	where $C_1(p,\alpha):= (2+2^{\frac{1+\alpha}{p}+1})C(p,\alpha)$. This implies that
	$$
	|f'(z)|\leq C_1(p,\alpha) y^{-\frac{1+\alpha}{p}-1} \|f\|_{\Bergman}
	$$
	for all $z=x+iy\in\mathbb C_+$,	and thus,
	$$\begin{aligned}
	\left|\mathscr H_{\widetilde{\varphi}}(f')(z)\right| &\leq \int_0^\infty |f'(z/t)| \frac{|\varphi(t)|}{t^2} dt\\
	&\leq C_1(p,\alpha) y^{-\frac{1+\alpha}{p}-1}\|f\|_{\Bergman} \int_0^\infty t^{\frac{1+\alpha}{p} -1} |\varphi(t)| dt<\infty.
	\end{aligned}$$
	Therefore, using the Lebesgue dominated convergence theorem, we obtain 
	$$(\Hau f)'(z)= \int_0^\infty f'(z/t)\frac{\varphi(t)}{t^2} dt$$
	for all $z\in \bC_+$. This proves that $\Hau f$ is holomorphic on $\mathbb C_+$, moreover,
	$$(\Hau f)'= \mathscr H_{\widetilde{\varphi}} (f').$$

	(ii) By the Minkowski inequality, for every $f\in\Bergman$, 
	$$\begin{aligned}
		\|\Hau f\|_{\Bergman}&=\left[\int_{\bC_+} \left|\int_0^\infty f\left(\frac{z}{t}\right) \frac{\varphi(t)}{t} dt\right|^p dA_\alpha(z)\right]^{1/p}\\
		&\leq \|f\|_{\Bergman} \int_0^\infty t^{\frac{1+\alpha}{p}-1}|\varphi(t)|dt.
	\end{aligned}$$
	This implies that $\Hau$ is bounded on $\Bergman$, moreover,
	\begin{equation}\label{15:33, 01/7/2021}
		\|\Hau\|_{\Bergman\to\Bergman}\leq \int_0^\infty t^{\frac{1+\alpha}p-1}|\varphi(t)| dt.
	\end{equation}
	
	 Let $0<\delta<1$ and let $\varphi_\delta$ be as in (\ref{14:48, 21/11/2021}). Then, it follows from (\ref{15:33, 01/7/2021}) that
	$$
	\|\mathscr H_{\varphi_\delta}\|_{\Bergman\to\Bergman} \leq \int_0^\infty t^{\frac{1+\alpha}{p}-1} |\varphi_\delta(t)| dt = \int_\delta^{1/\delta} t^{\frac{1+\alpha}{p}-1}|\varphi(t)| dt  <\infty
	$$
	and
	\begin{align}\label{18:28, 01/7/2021}
		\|\Hau -\mathscr H_{\varphi_\delta}\|_{\Bergman\to\Bergman} &\leq \int_0^\infty t^{\frac{1+\alpha}{p}-1} |\varphi(t)-\varphi_\delta(t)| dt\nonumber\\
		&= \int_0^\delta t^{\frac{1+\alpha}{p}-1} |\varphi(t)| dt + \int_{1/\delta}^\infty t^{\frac{1+\alpha}{p}-1} |\varphi(t)| dt<\infty.
	\end{align}
	
	For any $0<\varepsilon<1$, let  
	$$f_\varepsilon(z):= \varepsilon^{\frac{1+\alpha}{p} +\varepsilon} \Phi_\varepsilon(\varepsilon z)=\frac{1}{(z+i)^{\frac{1+\alpha}{p} +\varepsilon}},\quad \forall\, z\in\bC_+.$$
	Then, by Remark \ref{10:09, 08/7/2021}(i), we have
	\begin{equation}\label{18:04, 01/7/2021}
	\|f_\varepsilon\|_{\Bergman} \sim \varepsilon^{-1/p},
	\end{equation}
	where the constants $C>0$ are independent of $\varepsilon$. Moreover,
	$$\mathscr H_{\varphi_\delta}(f_\varepsilon)(z) - f_\varepsilon(z)\int_0^\infty t^{\frac{1+\alpha}{p}-1}\varphi_\delta(t)dt = \int_{\delta}^{1/\delta}[\phi_{\varepsilon,z}(t)-\phi_{\varepsilon,z}(1)] t^{\frac{1+\alpha}{p}-1} \varphi(t)dt,$$
	where $\phi_{\varepsilon,z}(t):=  \frac{t^{\varepsilon}}{(z+ ti)^{\frac{1+\alpha}{p} +\varepsilon}}$. For every $t\in (\delta,1/\delta)$ and $z\in\bC_+$, the Lagrange mean value theorem gives
	\begin{eqnarray*}
		|\phi_{\varepsilon,z}(t)-\phi_{\varepsilon,z}(1)| &\leq& |t-1|\sup_{s\in (\delta,1/\delta)}|\phi_{\varepsilon,z}'(s)|\\
		&\leq& (\delta^{-1}-1)\left(\frac{\varepsilon \delta^{\varepsilon-1}}{|z+\delta i|^{\frac{1+\alpha}{p}+\varepsilon}} + \frac{(\frac{1+\alpha}{p}+\varepsilon)\delta^{-\varepsilon}}{|z+\delta i|^{1+\frac{1+\alpha}{p}+\varepsilon}}\right)\\
		&\leq& \frac{\varepsilon \delta^{-(2+\frac{1+\alpha}{p})}}{|z+ i|^{\frac{1+\alpha}{p}+\varepsilon}} + \frac{(\frac{1+\alpha}{p}+1) \delta^{-(4+\frac{1+\alpha}{p})}}{|z+ i|^{1+\frac{1+\alpha}{p}}}.
	\end{eqnarray*}
This, together with (\ref{18:04, 01/7/2021}),  yields
\begin{eqnarray*}\label{an estimate for the norm}
&&\frac{\left\|\mathscr H_{\varphi_\delta}(f_\varepsilon) - f_\varepsilon \int_0^\infty t^{\frac{1+\alpha}{p}-1}\varphi_\delta(t)dt \right\|_{\Bergman}}{\|f_\varepsilon\|_{\Bergman}}\\
&\lesssim& \int_{\delta}^{1/\delta} t^{\frac{1+\alpha}{p}-1}|\varphi(t)|dt\left[\varepsilon \delta^{-(2+\frac{1+\alpha}{p})} + \varepsilon^{1/p} \delta^{-(4+\frac{1+\alpha}{p})}\right] \to 0\nonumber
\end{eqnarray*}
as $\varepsilon\to 0$. As a consequence, we obtain
$$\left|\int_{\delta}^{1/\delta} t^{\frac{1+\alpha}{p}-1}\varphi(t)dt \right|= \left|\int_{0}^{\infty} t^{\frac{1+\alpha}{p}-1}\varphi_\delta(t)dt\right|\leq \|\mathscr H_{\varphi_\delta}\|_{\Bergman\to\Bergman}.$$
Therefore, by (\ref{18:28, 01/7/2021}) and $\int_0^\infty t^{\frac{1+\alpha}{p}-1} |\varphi(t)| dt<\infty$,  
$$\begin{aligned}
	&\quad\|\Hau\|_{\Bergman\to\Bergman}\\
	&\geq \left|\int_0^\infty t^{\frac{1+\alpha}{p}-1}\varphi(t)dt\right| - 2 \left(\int_0^\delta t^{\frac{1+\alpha}{p}-1}|\varphi(t)|dt + \int_{1/\delta}^\infty t^{\frac{1+\alpha}{p}-1}|\varphi(t)|dt \right)\to \left|\int_0^\infty t^{\frac{1+\alpha}{p}-1}\varphi(t)dt\right|
\end{aligned}
$$
as $\delta \to 0$. This, together with (\ref{15:33, 01/7/2021}), allows us to conclude that
$$\left|\int_0^\infty t^{\frac{1+\alpha}{p}-1}\varphi(t)dt\right|\leq \|\Hau\|_{\Bergman\to\Bergman}\leq \int_0^\infty t^{\frac{1+\alpha}{p}-1}|\varphi(t)| dt.$$
	
\end{proof}

Now we are ready to give the proofs of Theorems \ref{the main theorem 1} and \ref{16:22, 26/4/2025}.

\begin{proof}[Proof of Theorem \ref{the main theorem 1}]
	By Lemma \ref{10:57, 03/7/2021}, it suffices to show that if $\varphi\geq 0$ and $\Hau$ is bounded on $\Bergman$, then 
	\begin{equation}\label{09:46, 23/11/2021}
		\int_{0}^{\infty} t^{\frac{1+\alpha}{p}-1} \varphi(t) dt<\infty.
	\end{equation}
	Indeed, for any $0<\varepsilon<\frac{\pi}{2}$ and $-\pi\leq a<b \leq \pi$, let $\Phi_\varepsilon$, $A_{(a,b)}$ and $S_{(a,b)}$ be as in (\ref{key test function}), (\ref{15:00, 05/7/2021}) and (\ref{15:02, 05/7/2021}), respectively. Now let us consider the following two cases.
	
	{\it Case 1:} $\frac{1+\alpha}{p}\notin \{2n: n\in\bN\}$. Then, by Lemma \ref{14:44, 08/7/2021}(i), there exist two constants $\varepsilon(p,\alpha)\in (0,\frac{\pi}{2})$ and $C(p,\alpha)>0$ such that 
	$$\begin{cases}
	\mbox{$\Im(\Phi_\varepsilon)$ has constant sign on $A_{(\frac{\pi}{2}-\varepsilon(p,\alpha), \frac{\pi}{2})}$ for all $0<\varepsilon<\varepsilon(p,\alpha)$},\\
	\mbox{$|\Im\left[\Phi_\varepsilon(z)\right]|\geq C(p,\alpha) \left|\Phi_\varepsilon(z)\right|$ for all $z\in A_{(\frac{\pi}{2}-\varepsilon(p,\alpha), \frac{\pi}{2})}$}.
	\end{cases}$$
	Therefore, by $\varphi$ being a nonnegative function on $(0,\infty)$,
	$$\left|\int_0^\infty \Im \left[\Phi_\varepsilon\left(\frac{z}{t}\right)\right]\frac{\varphi(t)}tdt\right|= \int_0^\infty \left|\Im\left[\Phi_\varepsilon\left(\frac{z}{t}\right)\right]\right|\frac{\varphi(t)}{t} dt$$
	and	
	$$\begin{aligned}
		\left|\Im\left[\Phi_\varepsilon\left(\frac{z}{t}\right)\right]\right|\geq C(p,\alpha) \left|\Phi_\varepsilon\left(\frac{z}{t}\right)\right| &= C(p,\alpha) \frac{t^{\frac{1+\alpha}{p}+\varepsilon}}{\sqrt{r^2+2(r\sin\theta)t\varepsilon +(t\varepsilon)^2}^{\frac{1+\alpha}{p}+\varepsilon}}\\
		&\geq C(p,\alpha) \frac{t^{\frac{1+\alpha}{p}+\varepsilon}}{(r+t\varepsilon)^{\frac{1+\alpha}{p}+\varepsilon}}
	\end{aligned}$$
	for every $0<\varepsilon<\varepsilon(p,\alpha)$, $z=r e^{i\theta}\in A_{(\frac{\pi}{2}-\varepsilon(p,\alpha),\frac{\pi}{2})}$ and $t>0$. Thus, using polar coordinates over the truncated sector $S_{(\frac{\pi}{2}-\varepsilon(p,\alpha),\frac{\pi}{2})}$,
	$$\begin{aligned}
		\|{\mathscr H}_\varphi\left(\Phi_{\varepsilon}\right)\|^p_{{\mathcal A}_{\alpha}^p\left(\mathbb C_+\right)}&\geq \|\Im\left[{\mathscr H}_\varphi\left(\Phi_{\varepsilon}\right)\right]\|^p_{{\mathcal A}_{\alpha}^p\left(\mathbb C_+\right)}\\
		&= \|{\mathscr H}_\varphi\left(\Im(\Phi_{\varepsilon})\right)\|^p_{{\mathcal A}_{\alpha}^p\left(\mathbb C_+\right)}\\
		&\geq \int_{S_{(\frac{\pi}{2}-\varepsilon(p,\alpha),\frac{\pi}{2})}} \left|\int_0^\infty \Im\left[\Phi_\varepsilon\left(\frac{z}{t}\right)\right]\frac{\varphi(t)}tdt\right|^{p}dA_{\alpha}(z)\\
		&\gsim \int_{S_{(\frac{\pi}{2}-\varepsilon(p,\alpha),\frac{\pi}{2})}} \left(\int_0^\infty \frac{t^{\frac{1+\alpha}{p}+\varepsilon}}{(r+t\varepsilon)^{\frac{1+\alpha}{p}+\varepsilon}}\frac{\varphi(t)}tdt\right)^{p}dA_{\alpha}(z)\\
		&\geq \int_1^\infty\int_{\frac{\pi}{2}-\varepsilon(p,\alpha)}^{\frac{\pi}2}\left(\int_0^{1/\varepsilon} \frac{t^{\frac{1+\alpha}{p}+\varepsilon}}{(r+t\varepsilon)^{\frac{1+\alpha}{p}+\varepsilon}}  \frac{\varphi(t)}{t}dt \right)^p r^{\alpha+1}(\sin\theta)^{\alpha}  dr d\theta\\
		&\gsim \int_1^\infty\frac{1}{r^{1+p\varepsilon}}dr\int_{\frac{\pi}{2}-\varepsilon(p,\alpha)}^{\frac{\pi}{2}}  (\sin\theta)^{\alpha} d\theta \left(\int_0^{1/\varepsilon} t^{\frac{1+\alpha}{p}-1+\varepsilon} \varphi(t) dt \right)^p \\
		&\gsim \varepsilon^{-1} \left(\int_0^{1/\varepsilon} t^{\frac{1+\alpha}{p}-1+\varepsilon} \varphi(t) dt \right)^p, 
	\end{aligned}$$
	where the positive constants $C$ depend only on $p$ and $\alpha$. Thus, by Remark \ref{10:09, 08/7/2021}(i),
	$$\int_0^{1/\varepsilon} t^{\frac{1+\alpha}{p}-1+\varepsilon} \varphi(t) dt \lesssim \varepsilon^{-\varepsilon}\frac{\|{\mathscr H}_\varphi\left(\Phi_{\varepsilon}\right)\|_{{\mathcal A}_{\alpha}^p\left(\mathbb C_+\right)}}{\|\Phi_\varepsilon\|_{\Bergman}}\leq \varepsilon^{-\varepsilon} \|\Hau\|_{\Bergman\to\Bergman}.$$
	Letting $\varepsilon\to 0$, we obtain 
	$$\int_0^{\infty} t^{\frac{1+\alpha}{p}-1} \varphi(t) dt \lesssim \|\Hau\|_{\Bergman\to\Bergman}<\infty,$$
	which proves that (\ref{09:46, 23/11/2021}) holds.
	
	{\it Case 2:} $\frac{1+\alpha}{p}\in \{2n: n\in\bN\}$. Then, by Lemma \ref{14:44, 08/7/2021}(ii), there exist two constants $\varepsilon(p,\alpha)\in (0,\frac{\pi}{2})$ and $C(p,\alpha)>0$ such that 
	$$\begin{cases}
	\mbox{$\Re(\Phi_\varepsilon)$ has constant sign on $A_{(\frac{\pi}{2}-\varepsilon(p,\alpha), \frac{\pi}{2})}$ for all $0<\varepsilon<\varepsilon(p,\alpha)$},\\
	\mbox{$|\Re\left[\Phi_\varepsilon(z)\right]|\geq C(p,\alpha) \left|\Phi_\varepsilon(z)\right|$ for all $z\in A_{(\frac{\pi}{2}-\varepsilon(p,\alpha), \frac{\pi}{2})}$}.
	\end{cases}$$
	By a similar argument to Case 1 (replacing $\Im(\Phi_\varepsilon)$ by $\Re(\Phi_\varepsilon)$), we also get
	$$\int_0^{1/\varepsilon} t^{\frac{1+\alpha}{p}-1+\varepsilon} \varphi(t) dt \lesssim \varepsilon^{-\varepsilon}\frac{\|{\mathscr H}_\varphi\left(\Phi_{\varepsilon}\right)\|_{{\mathcal A}_{\alpha}^p\left(\mathbb C_+\right)}}{\|\Phi_\varepsilon\|_{\Bergman}}\leq \varepsilon^{-\varepsilon} \|\Hau\|_{\Bergman\to\Bergman}.$$
	Letting $\varepsilon\to 0$, we obtain 
	$$\int_0^{\infty} t^{\frac{1+\alpha}{p}-1} \varphi(t) dt \lesssim \|\Hau\|_{\Bergman\to\Bergman}<\infty$$
	which proves that (\ref{09:46, 23/11/2021}) holds, and thus ends the proof of Theorem \ref{the main theorem 1}.	
\end{proof}


\begin{proof}[Proof of Theorem \ref{16:22, 26/4/2025}]
	Let $f\in \mathcal D(\bC_+)$. Thanks to \cite[Proposition 1.3]{BBGNPR}, we have
	$$
	|f'(z)|\lesssim  y^{-1} \|f'\|_{\mathcal A^2_1(\bC_+)}= y^{-1} \|f\|_{\mathcal D(\bC_+)}
	$$
	for  all $z=x+iy\in\bC_+$. Therefore, the Lagrange mean value theorem yields
	$$\left|f\left(\frac{z}{t}\right)\right|\lesssim |f(z)| + \frac{|z|}{y}\left|\log\frac{1}{t}\right| \|f\|_{\mathcal D(\bC_+)}$$
	for all $z=x+iy\in\bC_+$ and all $t\in (0,\infty)$. Thus, 
	$$\begin{aligned}
		|\Hau f(z)|&\leq \int_0^\infty \left|f\left(\frac{z}{t}\right)\right| \left|\frac{\varphi(t)}{t}\right| dt\\
		&\lesssim |f(z)| \int_0^\infty \frac{|\varphi(t)|}{t} dt + \frac{|z|}{y} \|f\|_{\mathcal D(\bC_+)}\int_0^\infty \frac{|\log t|}{t} |\varphi(t)| dt<\infty.
	\end{aligned}$$
	This proves that $\Hau f$ is well-defined on $\bC_+$.
	
	On the other hand, it follows from Lemma \ref{10:57, 03/7/2021} and the Lebesgue dominated convergence theorem that $\mathscr H_{\widetilde{\varphi}}(f')(z)= \int_{0}^{\infty} f'(z/t)\frac{\varphi(t)}{t^2}dt$ is well-defined for all $z\in\bC_+$. Moreover, by Theorem \ref{07:43, 19/11/2021}(i),
	$$\begin{aligned}
		\|\mathscr H_{\varphi} f\|_{\mathcal D(\bC_+)} = \|(\mathscr H_{\varphi} f)'\|_{\mathcal A^2_1(\bC_+)}&= \|\mathscr H_{\widetilde{\varphi}}(f')\|_{\mathcal A^2_1(\bC_+)}\\
		&\leq \int_0^\infty |\widetilde{\varphi}(t)| dt \|f'\|_{\mathcal A^2_1(\bC_+)}=\int_0^\infty \frac{|\varphi(t)|}{t} dt \|f\|_{\mathcal D(\bC_+)}.
	\end{aligned}$$
	Thus $\Hau$ is bounded on $\mathcal D(\bC_+)$ and
	$$\|\Hau\|_{\mathcal D(\bC_+)\to \mathcal D(\bC_+)}\leq \int_0^\infty \frac{|\varphi(t)|}{t} dt.$$	
\end{proof}

\begin{remark}
	In Theorem \ref{16:22, 26/4/2025}, the condition  $\int_0^\infty \frac{|\log t|}{t} |\varphi(t)| dt<\infty$ can be replaced by the weaker condition $\int_0^\infty \frac{|\varphi(t)|}{t}  dt<\infty$ if one defines 
	$$\mathscr H_{\varphi} f:= \lim_{\delta\to 0} \mathscr H_{\varphi_\delta} f$$
	in $\mathcal D(\bC_+)$, where $\varphi_\delta$ is as in (\ref{14:48, 21/11/2021}).
\end{remark}

\section{Hausdorff operators $\Hau$ acting on $\Hardy$}

The main purpose of this section is to give the proof of Theorem \ref{the main theorem 2}. Given $1<r<\infty$, a locally integrable nonnegative function $w: \R \to [0,\infty)$ is said to belong to the {\it reverse H\"older class} $RH_r(\R)$ if
$$\left(\frac{1}{|I|}\int_I [w(x)]^r dx\right)^{1/r}\lesssim \frac{1}{|I|}\int_I w(x) dx$$
for all intervals $I\subset \R$. The following lemma is well-known (see the book of Garc\'ia-Cuerva and Rubio de Francia \cite[CHAPTER IV]{GR}). 

\begin{lemma}\label{15:39, 18/11/2021}
	\begin{enumerate}[\rm (i)]
		\item Let $w\in A_q(\R)$. Then, there exists $1<r(q,w)<\infty$, depending only $q$ and $w$, such that $w\in RH_{r(q,w)}(\R)$.
		
		\item Let $1\leq q<\infty$, $1<r<\infty$ and $w\in A_q(\R)\cap RH_r(\R)$. Then
		$$\left(\frac{|E|}{|I|}\right)^q \lesssim \frac{\int_E w(x)dx}{\int_I w(x)dx}\lesssim \left(\frac{|E|}{|I|}\right)^{(r-1)/r}$$
		for all intervals $I\subset \R$ and all measurable sets $E\subset I$.
	\end{enumerate}
\end{lemma}

Denote by 
$$w_\alpha(x):=|x|^\alpha\quad\mbox{and}\quad B(x,t):= \{y\in\R: |y-x|<t\}$$
for all $x\in\R$ and $t>0$. In what follows, for any bounded measurable set $E\subset \R$,
$$w_\alpha(E):=\int_E w_\alpha(x) dx= \int_E |x|^\alpha dx.$$

\begin{lemma}\label{15:55, 18/11/2021}
	Let $x_0+i y_0\in \bC_+$ and $x+i y\in D= \{z\in\bC: |z|<1\}$. Then
	$$w_\alpha\left(B\left(x_0+\frac{y_0}{2}x, y_0+ \frac{y_0}{2}y\right)\right)\sim w_\alpha\left(B(x_0,y_0)\right),$$
	where the constants $C>0$ are independent of $x_0, y_0, x, y$. 
\end{lemma}

\begin{proof}
	From $y_0>0$, it is easy to verify that
	$$B\left(x_0+\frac{y_0}{2}x, y_0+ \frac{y_0}{2}y\right)\subset B(x_0, 2y_0)\subset B\left(x_0+\frac{y_0}{2}x, 5\left(y_0+ \frac{y_0}{2}y\right)\right).$$
	This, together with Remark \ref{10:11, 18/11/2021} and Lemma \ref{15:39, 18/11/2021}, implies that
	$$w_\alpha\left(B\left(x_0+\frac{y_0}{2}x, y_0+ \frac{y_0}{2}y\right)\right)\sim w_\alpha\left(B(x_0,y_0)\right).$$
\end{proof}

\begin{lemma}\label{15:10, 15/7/2021}
	Let $1\leq p\leq\infty$ and $\alpha>-1$. Then
	\begin{equation}\label{Ga, Lemma II.1.2}
		|f(z)|=|f(x+iy)|\lesssim [w_\alpha(B(x,y))]^{-1/p} \|f\|_{\Hardy}
	\end{equation}
	and
	\begin{equation}\label{15:36, 14/7/2021}
		|f'(z)|=|f'(x+iy)|\lesssim \frac{[w_\alpha(B(x,y))]^{-1/p}}{y} \|f\|_{\Hardy}
	\end{equation}
	for all $f\in\Hardy$ and all $z=x+iy\in\bC_+$.
\end{lemma}

\begin{proof}
	It is obvious if $p=\infty$ or $f=0$. Otherwise, thanks to \cite[Lemma II.1.2]{Ga}, we have
	\begin{equation}\label{10:49, 15/7/2021}
		|f(z)|=|f(x+iy)|\leq C(p,\alpha) [w_\alpha(B(x,y))]^{-1/p} \|f\|_{\Hardy},
	\end{equation}
	where the constant $C(p,\alpha)>0$ depends only on $p$ and $\alpha$. This proves (\ref{Ga, Lemma II.1.2}).
	
	Now let us prove (\ref{15:36, 14/7/2021}). Indeed, for every $z_0=x_0+i y_0\in \bC_+$ and $z=x+i y\in \mathbb D$, by Lemma \ref{15:55, 18/11/2021}, there exist two constants $0<C_1(\alpha) < C_2(\alpha)$ such that
	$$C_1(\alpha) w_\alpha\left(B(x_0,y_0)\right)\leq w_\alpha\left(B\left(x_0+\frac{y_0}{2}x, y_0+ \frac{y_0}{2}y\right)\right)\leq C_2(\alpha) w_\alpha\left(B(x_0,y_0)\right).$$
	Therefore, by (\ref{10:49, 15/7/2021}), we get
	$$\left|f\left(z_0+\frac{y_0}{2}z\right)\right|\leq C(p,\alpha) [C_1(\alpha)]^{-1/p} [w_\alpha\left(B(x_0,y_0)\right)]^{-1/p} \|f\|_{\Hardy}.$$
	Thus, if we define the holomorphic function $g: \mathbb D\to \bC$ as follows
	$$g(z)=\frac{f(z_0+ \frac{y_0}{2}z)-f(z_0)}{(1+[C_1(\alpha)]^{-1/p})C(p,\alpha) [w_\alpha\left(B(x_0,y_0)\right)]^{-1/p} \|f\|_{\Hardy}},$$
	then, by using Lemma \ref{the Schwarz lemma}, we obtain
	$$\begin{aligned}
	|f'(z_0)|&=\frac{2}{y_0}(1+[C_1(\alpha)]^{-1/p})C(p,\alpha) [w_\alpha\left(B(x_0,y_0)\right)]^{-1/p} \|f\|_{\Hardy} |g'(0)|\\
	&\leq 2(1+[C_1(\alpha)]^{-1/p})C(p,\alpha) \frac{[w_\alpha\left(B(x_0,y_0)\right)]^{-1/p}}{y_0} \|f\|_{\Hardy}.
	\end{aligned}$$
	This proves (\ref{15:36, 14/7/2021}).
\end{proof}

\begin{lemma}\label{14:38, 14/7/2021}
	Let $1\leq p\leq \infty$, $\alpha>-1$, and let $\varphi$ be such that (\ref{main inequality}) holds. Then:
	\begin{enumerate}[\rm (i)]
		\item $\Hau f$ is well-defined and  holomorphic on $\mathbb C_+$ for all $f\in \Hardy$. Moreover, 
		$$(\Hau f)'= \mathscr H_{\widetilde{\varphi}} (f'),$$
		where $\widetilde{\varphi}$ is defined as in (\ref{07:18, 19/11/2021}).
		
		\item $\Hau$ is bounded on $\Hardy$, moreover,
		$$\left|\int_0^\infty t^{\frac{1+\alpha}{p}-1}\varphi(t)dt\right|\leq \|\Hau\|_{\Hardy\to\Hardy}\leq \int_0^\infty t^{\frac{1+\alpha}{p}-1}|\varphi(t)| dt.$$
	\end{enumerate}
\end{lemma}

\begin{proof}
	(i) By Lemma \ref{15:10, 15/7/2021}, we see that
	$$\begin{aligned}
	|\Hau f(z)| &\leq \int_0^\infty \left|f\left(\frac{z}{t}\right)\right| \frac{|\varphi(t)|}{t} dt\\
	&\lesssim  \int_0^\infty \left[w_\alpha\left(B\left(\frac{x}{t},\frac{y}{t}\right)\right)\right]^{-1/p} \|f\|_{\Hardy} \frac{|\varphi(t)|}{t}dt\\
	&=[w_\alpha(B(x,y))]^{-1/p} \|f\|_{\Hardy}\int_0^\infty  t^{\frac{1+\alpha}{p}-1} |\varphi(t)| dt<\infty,
	\end{aligned}$$
	which proves that $\Hau f$ is well-defined on $\mathbb C_+$. We also see that
	$$\begin{aligned}
	\left|\mathscr H_{\widetilde{\varphi}}(f')(z)\right| &\leq \int_0^\infty \left|f'\left(\frac{z}{t}\right)\right| \frac{|\varphi(t)|}{t^2} dt\\
	&\lesssim  \int_0^\infty \frac{\left[w_\alpha\left(B\left(\frac{x}{t},\frac{y}{t}\right)\right)\right]^{-1/p}}{\frac{y}{t}} \|f\|_{\Hardy} \frac{|\varphi(t)|}{t^2}dt\\
	&=\frac{[w_\alpha(B(x,y))]^{-1/p}}{y} \|f\|_{\Hardy}\int_0^\infty  t^{\frac{1+\alpha}{p}-1} |\varphi(t)| dt<\infty
	\end{aligned}$$	
	for all $z=x+iy\in\bC_+$ and, using the Lebesgue dominated convergence theorem, we obtain 
	$$(\Hau f)'(z)=\mathscr H_{\widetilde{\varphi}} (f')(z)= \int_0^\infty f'(z/t)\frac{\varphi(t)}{t^2} dt.$$
	
	(ii) By the Minkowski inequality, for every $f\in\Hardy$, 
	$$\begin{aligned}
	\|\Hau f\|_{\Hardy}&=\sup_{y>0}\left(\int_{-\infty}^\infty \left|\int_0^\infty f\left(\frac{x}{t}+i\frac{y}{t}\right)\frac{\varphi(t)}{t}dt\right|^p |x|^\alpha dx\right)^{1/p}\\
	&\leq \|f\|_{\Hardy} \int_0^\infty t^{\frac{1+\alpha}{p}-1}|\varphi(t)|dt.
	\end{aligned}$$
	This implies that $\Hau$ is bounded on $\Hardy$, moreover,
	\begin{equation}\label{15:45, 15/7/2021}
	\|\Hau\|_{\Hardy\to\Hardy}\leq \int_0^\infty t^{\frac{1+\alpha}p-1}|\varphi(t)| dt.
	\end{equation}

	Let $0<\delta<1$ and let $\varphi_\delta$ be as in (\ref{14:48, 21/11/2021}). For any $0<\varepsilon<1$, denote by
	$$f_\varepsilon(z):= \varepsilon^{\frac{1+\alpha}{p} +\varepsilon} \Phi_\varepsilon(\varepsilon z)=\frac{1}{(z+i)^{\frac{1+\alpha}{p} +\varepsilon}},\quad \forall\, z\in\bC_+.$$
	Then, by Remark \ref{10:09, 08/7/2021}(i), 
	 \begin{equation*}\label{17:27, 15/7/2021}
	 \|f_\varepsilon\|_{\Hardy} \sim \varepsilon^{-1/p}.
	 \end{equation*}
	 A similar argument to the proof of Lemma \ref{10:57, 03/7/2021}(ii) yields
	 \begin{equation}\label{17:34, 22/11/2021}
	 	\begin{aligned}
	 		&\frac{\left\|\mathscr H_{\varphi_\delta}(f_\varepsilon) - f_\varepsilon \int_0^\infty t^{\frac{1+\alpha}{p}-1}\varphi_\delta(t)dt \right\|_{\Hardy}}{\|f_\varepsilon\|_{\Hardy}}\\
	 		&\lesssim \int_{\delta}^{1/\delta} t^{\frac{1+\alpha}{p}-1}\varphi(t)dt\left[\varepsilon \delta^{-(2+\frac{1+\alpha}{p})} + \varepsilon^{1/p} \delta^{-(4+\frac{1+\alpha}{p})}\right] \to 0
	 	\end{aligned}
	 \end{equation}
	as $\varepsilon\to 0$. As a consequence, we obtain
	$$\left|\int_{\delta}^{1/\delta} t^{\frac{1+\alpha}{p}-1}\varphi(t)dt \right|= \left|\int_{0}^{\infty} t^{\frac{1+\alpha}{p}-1}\varphi_\delta(t)dt\right|\leq \|\mathscr H_{\varphi_\delta}\|_{\Hardy\to\Hardy},$$
	and thus, 
	$$\begin{aligned}
		&\;\quad\|\Hau\|_{\Hardy\to\Hardy}\\
		&\geq \|\mathscr H_{\varphi_\delta}\|_{\Hardy\to\Hardy}- \|\mathscr H_{\varphi}-\mathscr H_{\varphi_\delta}\|_{\Hardy\to\Hardy}\\
		&\geq \left|\int_0^\infty t^{\frac{1+\alpha}{p}-1}\varphi(t)dt\right| - 2 \left(\int_0^\delta t^{\frac{1+\alpha}{p}-1}|\varphi(t)|dt + \int_{1/\delta}^\infty t^{\frac{1+\alpha}{p}-1}|\varphi(t)|dt \right)\to \left|\int_0^\infty t^{\frac{1+\alpha}{p}-1}\varphi(t)dt\right|
	\end{aligned}$$
	 as $\delta \to 0$. This, together with (\ref{15:45, 15/7/2021}), allows us to conclude that
	$$\left|\int_0^\infty t^{\frac{1+\alpha}{p}-1}\varphi(t)dt\right|\leq \|\Hau\|_{\Hardy\to\Hardy}\leq \int_0^\infty t^{\frac{1+\alpha}{p}-1}|\varphi(t)|dt.$$
\end{proof}

Now we are ready to give the proof of Theorem \ref{the main theorem 2}.

\begin{proof}[Proof of Theorem \ref{the main theorem 2}]
	By Lemma \ref{14:38, 14/7/2021}, it suffices to show that if $\varphi\geq 0$ and $\Hau$ is bounded on $\Hardy$, then 
	\begin{equation}\label{11:30, 23/11/2021}
		\int_{0}^{\infty} t^{\frac{1+\alpha}{p}-1} \varphi(t) dt<\infty.
	\end{equation}	
	
	{\it Case 1}: $p=\infty$. From $f=1\in \mathcal H^\infty_{|\cdot|^\alpha}(\bC_+)$ and $\Hau$ is bounded on $\mathcal H^\infty_{|\cdot|^\alpha}(\bC_+)$, we obtain that $\int_{0}^{\infty} t^{-1}\varphi(t)dt =\Hau f$ is finite.
	
	{\it Case 2}: $1\leq p<\infty$.	A similar argument to the proof of Theorem \ref{the main theorem 1}(ii), we also get
	$$\int_0^{1/\varepsilon} t^{\frac{1+\alpha}{p}-1+\varepsilon} \varphi(t) dt \lesssim \varepsilon^{-\varepsilon}\frac{\|{\mathscr H}_\varphi\left(\Phi_{\varepsilon}\right)\|_{\Hardy}}{\|\Phi_\varepsilon\|_{\Hardy}}\leq \varepsilon^{-\varepsilon} \|\Hau\|_{\Hardy\to\Hardy}.$$
	Letting $\varepsilon\to 0$, we obtain 
	$$\int_0^{\infty} t^{\frac{1+\alpha}{p}-1} \varphi(t) dt \lesssim \|\Hau\|_{\Hardy\to\Hardy}<\infty.$$
	This proves that (\ref{11:30, 23/11/2021}) holds, and thus ends the proof of Theorem \ref{the main theorem 2}.
\end{proof}

\section{Some applications to the real version of $\mathscr H_\varphi$}

In this section, we give some applications of Theorem  \ref{the main theorem 2} in studying the Hausdorff operators $\mathcal H_\varphi$ on the real function spaces. Our primary goal is to provide the proofs of Theorems \ref{16:37, 26/4/2025} and \ref{16:42, 26/4/2025}. Let us begin by presenting some interesting examples of Hausdorff operators (see \cite{AS, BBMM, LMPS, MO}).

\begin{example}
	Let $\nu\in\bC$ be such that  $\Re(\nu)>0$. Following \cite{AS, BBMM}, we define the Ces\`aro-like operator as
	$$\mathcal C_\nu f(z)=\frac{1}{z^\nu}\int_0^z f(\zeta) \zeta^{\nu-1} d\zeta,\quad z\in \bC_+.$$
	This operator is a special case of the Hausdorff operator $\Hau$ associated with the kernel $\varphi(t)=\frac{1}{t^{\nu}} \chi_{(1,\infty)}(t)$ for all $t\in (0,\infty)$.
\end{example}

\begin{example}
	Let $\beta>0$. Following \cite{LMPS}, we define the generalized Ces\`aro operator as 
	$$\mathscr C_\beta f(x)=\beta \int_0^1 f(tx) (1-t)^{\beta-1} dt,\quad x\in\R.$$
	This operator is a special case of the real Hausdorff operator $\H_\varphi$ associated with the kernel $\varphi(t)=\frac{\beta (t-1)^{\beta-1}}{t^{\beta}} \chi_{(1,\infty)}(t)$ for all $t\in (0,\infty)$.
\end{example}

\begin{example}
	Let $a: (0,\infty)\to (0,\infty)$ be a locally integrable function. Following \cite{CO}, we define the operator $S_a$ as
	$$S_a f(x)= \int_{0}^{\infty} f(tx) a(t) dt,\quad x\in\R.$$
	This operator is a special case of the real Hausdorff operator $\H_\varphi$ associated with the kernel $\varphi(t)=\frac{a(\frac{1}{t})}{t}$ for all $t\in (0,\infty)$. When $a(t)= \frac{t^{\mu-\beta}}{(t+1)^{\mu}}$, the operator $S_a$ becomes the generalized Stieltjes operator $\mathcal S_{\beta, \mu}$ as studied in \cite{MO}.
\end{example}

We first recall a lemma due to Garc\'ia-Cuerva \cite[THEOREM II.1.1]{Ga}.

\begin{lemma}\label{10:18, 21/11/2021}
	Let $0<p<\infty$, $\alpha>-1$. Then, 
		$$\|f^*\|_{L^p_{|\cdot|^\alpha}(\R)}\sim \|f\|_{\Hardy}\quad\mbox{and}\quad \lim_{y\to 0}\|f(\cdot+iy)- f^*\|_{L^p_{|\cdot|^\alpha}(\R)}=0$$
		for all $f\in \Hardy$.
\end{lemma}

In order to prove Theorems \ref{16:37, 26/4/2025} and \ref{16:42, 26/4/2025}, we also need the following two lemmas.

\begin{lemma}\label{17:12, 22/11/2021}
	Let $1\leq p\leq \infty$, $\alpha>-1$, and let $\varphi$ be a measurable function on $(0,\infty)$ such that (\ref{main inequality}) holds. Then, $\H_\varphi$ is bounded on $L^p_{|\cdot|^\alpha}(\R)$, moreover,
	$$\|\H_\varphi\|_{L^p_{|\cdot|^\alpha}(\R)\to L^p_{|\cdot|^\alpha}(\R)}\leq \int_{0}^{\infty} t^{\frac{1+\alpha}{p}-1} |\varphi(t)| dt.$$
\end{lemma}

\begin{proof}
	The case $p=\infty$ is trivial. Otherwise, the Minkowski inequality yields
	$$\|\mathcal H_\varphi f\|_{L^p_{|\cdot|^\alpha}(\R)}= \left(\int_{\R} \left|\int_{0}^{\infty} f\left(\frac{x}{t}\right)\frac{\varphi(t)}{t}dt\right|^p |x|^\alpha dx\right)^{1/p}\leq \int_{0}^{\infty} t^{\frac{1+\alpha}{p}-1} |\varphi(t)| dt \;\|f\|_{L^p_{|\cdot|^\alpha}(\R)}$$
	for all $f\in L^p_{|\cdot|^\alpha}(\R)$. 
	This proves that $\H_\varphi$ is bounded on $L^p_{|\cdot|^\alpha}(\R)$, moreover,
	$$\|\H_\varphi\|_{L^p_{|\cdot|^\alpha}(\R)\to L^p_{|\cdot|^\alpha}(\R)}\leq \int_{0}^{\infty} t^{\frac{1+\alpha}{p}-1} |\varphi(t)| dt.$$
\end{proof}

\begin{lemma}\label{07:49, 21/11/2021}
	Let $\alpha>-1$, $0<\delta<1$ and $x\in \R$. Then, there is a constant $C(\alpha, \delta, x)>0$ such that, for every $\delta<t<1/\delta$,
	$$w_\alpha\left(B\left(x/t,\delta\right)\right)=\int_{B\left(x/t,\delta\right)} |u|^\alpha du\geq C(\alpha, \delta, x).$$
\end{lemma}

\begin{proof}
	Since $\alpha>-1$, we see that the function
	$$f_{\alpha, \delta, x}(s):= \int_{B(s,\delta)} |u|^\alpha du, \quad s\in [-|x|/\delta, |x|/\delta],$$
	is continuous on the compact set $[-|x|/\delta, |x|/\delta]$. Therefore, for every $\delta<t<1/\delta$,
	$$w_\alpha\left(B\left(x/t,\delta\right)\right)= f_{\alpha, \delta, x}(x/t)\geq \min_{s\in [-|x|/\delta, |x|/\delta]} f_{\alpha, \delta, x}(s)=: C(\alpha, \delta, x)>0.$$
\end{proof}

Now we are ready to give the proofs of Theorems \ref{16:37, 26/4/2025} and \ref{16:42, 26/4/2025}.

\begin{proof}[Proof of Theorem \ref{16:37, 26/4/2025}]
	{\it Case 1:} $p=\infty$. For any $x\in \R$, by $f\in \H^\infty_{|\cdot|^\alpha}(\bC_+)$ and $\int_{0}^{\infty} t^{-1}|\varphi(t)|dt<\infty$, the Lebesgue dominated convergence theorem and (\ref{22:32, 20/11/2021}) yield
	$$(\Hau f)^*(x)= \lim_{y\to 0} \int_0^\infty f\left(\frac{x}{t}+\frac{y}{t}i\right)\frac{\varphi(t)}{t}dt=\int_0^\infty f^*\left(\frac{x}{t}\right)\frac{\varphi(t)}{t}dt=\mathcal H_\varphi(f^*)(x).$$
	
	{\it Case 2:} $1\leq p<\infty$. For any $0<\delta<1$, let $\varphi_\delta$ be as in (\ref{14:48, 21/11/2021}), and define
	$$f_\delta(z):= f(z+\delta i),\quad z\in\bC_+.$$
	Then, 
	\begin{equation}\label{08:02, 21/11/2021}
		\int_{0}^{\infty} \left|\frac{\varphi_\delta(t)}{t}\right| dt\leq \delta^{-\frac{1+\alpha}{p}}\int_{0}^{\infty} t^{\frac{1+\alpha}{p}-1} |\varphi(t)| dt<\infty.
	\end{equation}
	Moreover, for every $x\in \R$, it follows from Lemmas \ref{15:10, 15/7/2021} and \ref{07:49, 21/11/2021} that there are two constant $C(p,\alpha), C(\alpha, \delta, x)>0$ for which
	\begin{eqnarray*}
		\left|f_\delta\left(\frac{x}{t}+\frac{y}{t}i\right)\right|&=& \left|f\left(x/t+\left(y/t+\delta\right)i\right)\right|\\
		&\leq& C(p,\alpha) \left[w_\alpha(B(x/t,y/t+ \delta))\right]^{-1/p}\|f\|_{\Hardy}\\
		&\leq& C(p,\alpha) \left[w_\alpha(B(x/t,\delta))\right]^{-1/p}\|f\|_{\Hardy}\\
		&\leq& C(p,\alpha) \left[C(\alpha, \delta, x)\right]^{-1/p}\|f\|_{\Hardy}
	\end{eqnarray*}
	for all $y>0$ and all $t\in (\delta,1/\delta)$. This, together with (\ref{08:02, 21/11/2021}) and the Lebesgue dominated convergence theorem, allows us to conclude that
	$$(\mathscr H_{\varphi_\delta} (f_\delta))^*(x)= \lim_{y\to 0} \int_0^\infty f_\delta\left(\frac{x}{t}+\frac{y}{t}i\right)\frac{\varphi_\delta(t)}{t}dt =\int_0^\infty (f_\delta)^*\left(\frac{x}{t}\right)\frac{\varphi_\delta(t)}{t}dt=\mathcal H_{\varphi_\delta}((f_\delta)^*)(x).$$
	Therefore, by Theorem \ref{the main theorem 2}(i), Lemmas \ref{10:18, 21/11/2021} and \ref{17:12, 22/11/2021}, we obtain that  
	\begin{eqnarray*}
		\|(\Hau f)^* - \mathcal H_\varphi(f^*)\|_{L^p_{|\cdot|^\alpha}(\R)}
		&\leq& \|\mathscr H_{\varphi-\varphi_\delta} f\|_{\Hardy} + \|\mathcal H_{\varphi-\varphi_\delta}(f^*)\|_{L^p_{|\cdot|^\alpha}(\R)} + \\
		&&  +\|\mathscr H_{\varphi_\delta}(f- f_\delta)\|_{\Hardy} + \|\mathcal H_{\varphi_\delta}(f^* - (f_\delta)^*)\|_{L^p_{|\cdot|^\alpha}(\R)}\\
		&\lesssim& \|f\|_{\Hardy} \int_0^\infty t^{\frac{1+\alpha}{p}-1}|\varphi(t)-\varphi_\delta(t)|dt +\\
		&&+ \|f^*-(f_\delta)^*\|_{L^p_{|\cdot|^\alpha}(\R)}\int_0^\infty t^{\frac{1+\alpha}{p}-1}|\varphi_\delta(t)|dt\\
		&\lesssim& \|f\|_{\Hardy} \left[\int_0^\delta t^{\frac{1+\alpha}{p}-1}|\varphi(t)|dt + \int_{1/\delta}^\infty t^{\frac{1+\alpha}{p}-1}|\varphi(t)|dt \right] +\\
		&&+ \|f^*- f(\cdot+\delta i)\|_{L^p_{|\cdot|^\alpha}(\R)}\int_0^\infty t^{\frac{1+\alpha}{p}-1}|\varphi(t)|dt\to 0
	\end{eqnarray*}
	as  $\delta\to 0$. This completes the proof of Theorem \ref{16:37, 26/4/2025}. 
	
\end{proof}

\begin{proof}[Proof of Theorem \ref{16:42, 26/4/2025}]
	(i) By Lemma \ref{17:12, 22/11/2021}, it suffices to prove that
	\begin{equation}\label{17:37, 22/11/2021}
		\left|\int_{0}^{\infty} t^{\frac{1+\alpha}{p}-1} \varphi(t) dt\right|\leq \|\H_\varphi\|_{L^p_{|\cdot|^\alpha}(\R)\to L^p_{|\cdot|^\alpha}(\R)}.
	\end{equation}
	Indeed, for any $0<\varepsilon<1$, we define $f_\varepsilon: \bC_+\to \bC$ as follows
	$$f_\varepsilon(z)= \varepsilon^{\frac{1+\alpha}{p}+\varepsilon} \Phi_\varepsilon(\varepsilon z)=\frac{1}{(z+i)^{\frac{1+\alpha}{p}+\varepsilon}}.$$
	Then, by Lemma \ref{10:18, 21/11/2021} and Remark \ref{10:09, 08/7/2021}(i),
	$$\|(f_\varepsilon)^*\|_{L^p_{|\cdot|^\alpha}(\R)}\sim \|f_\varepsilon\|_{\Hardy}\sim \varepsilon^{-1/p}.$$
	 Therefore,  by Theorem \ref{16:37, 26/4/2025}, Lemma \ref{10:18, 21/11/2021} and (\ref{17:34, 22/11/2021}), 
	$$\frac{\left\|\H_{\varphi_\delta}((f_\varepsilon)^*)- (f_{\varepsilon})^* \int_{0}^{\infty} t^{\frac{1+\alpha}{p}-1} \varphi_\delta(t)dt\right\|_{L^p_{|\cdot|^\alpha}(\R)}}{\|(f_\varepsilon)^*\|_{L^p_{|\cdot|^\alpha}(\R)}}\sim\frac{\left\|\mathscr H_{\varphi_\delta}(f_\varepsilon) - f_\varepsilon \int_{0}^{\infty} t^{\frac{1+\alpha}{p}-1} \varphi_\delta(t)dt\right\|_{\Hardy}}{\|f_\varepsilon\|_{\Hardy}} \to 0$$
	as $\varepsilon\to 0$. As a consequence, we obtain
	$$\left|\int_{\delta}^{1/\delta} t^{\frac{1+\alpha}{p}-1} \varphi(t) dt\right|= \left|\int_{0}^{\infty} t^{\frac{1+\alpha}{p}-1} \varphi_\delta(t)dt\right|\leq \|\H_{\varphi_\delta}\|_{L^p_{|\cdot|^\alpha}(\R)\to L^p_{|\cdot|^\alpha}(\R)},$$
	and thus, by $\int_{0}^{\infty} t^{\frac{1+\alpha}{p}-1} |\varphi(t)| dt<\infty$,
	$$\begin{aligned}
		&\quad\;\|\H_\varphi\|_{L^p_{|\cdot|^\alpha}(\R)\to L^p_{|\cdot|^\alpha}(\R)}\\
		&\geq \|\H_{\varphi_\delta}\|_{L^p_{|\cdot|^\alpha}(\R)\to L^p_{|\cdot|^\alpha}(\R)} -\|\H_{\varphi-\varphi_\delta}\|_{L^p_{|\cdot|^\alpha}(\R)\to L^p_{|\cdot|^\alpha}(\R)}\\
		& \geq\left|\int_{0}^{\infty} t^{\frac{1+\alpha}{p}-1} \varphi(t) dt\right| - 2 \left[\int_{0}^{\delta} t^{\frac{1+\alpha}{p}-1} |\varphi(t)| dt + \int_{1/\delta}^{\infty} t^{\frac{1+\alpha}{p}-1} |\varphi(t)| dt\right]\to \left|\int_{0}^{\infty} t^{\frac{1+\alpha}{p}-1} \varphi(t) dt\right|
	\end{aligned}$$
	as $\delta\to 0$. This proves that (\ref{17:37, 22/11/2021}) holds.
	
	(ii) From $\alpha>-1$, it follows that there exists $q\in (1,\infty)$ for which 
	$$\frac{1+\alpha}{p}> \frac{1}{q}.$$
	This gives
	$$\int_{0}^{\infty} t^{1/q-1} |\varphi_\delta(t)| dt= \int_{\delta}^{1/\delta} t^{1/q-1}|\varphi(t)|dt\leq \delta^{\frac{1}{q}-\frac{1+\alpha}{p}} \int_{0}^{\infty} t^{\frac{1+\alpha}{p}-1} |\varphi(t)| dt<\infty.$$
	Therefore, thanks to \cite[Theorem 3.1]{HKQ20}, we obtain
	$$H(\H_{\varphi_\delta} f) = \H_{\varphi_\delta}(H(f))$$
	for all $f\in L^q(\R)\cap L^p_{|\cdot|^\alpha}(\R)$. Thus, by Lemma \ref{17:12, 22/11/2021}, Remark \ref{10:11, 18/11/2021} and $\int_{0}^{\infty} t^{\frac{1+\alpha}{p}-1} |\varphi(t)| dt<\infty$, we see that
	$$\begin{aligned}
		&\quad\; \|H(\H_{\varphi} f)- \H_{\varphi}(H(f))\|_{L^p_{|\cdot|^\alpha}(\R)}\\
		&\leq \|H(\H_{\varphi} f)-H(\H_{\varphi_\delta} f)\|_{L^p_{|\cdot|^\alpha}(\R)} + \|\H_{\varphi}(H(f))- \H_{\varphi_\delta}(H(f))\|_{L^p_{|\cdot|^\alpha}(\R)}\\
		& \leq 2 \|H\|_{L^p_{|\cdot|^\alpha}(\R)\to L^p_{|\cdot|^\alpha}(\R)} \|f\|_{L^p_{|\cdot|^\alpha}(\R)}  \left[\int_{0}^{\delta} t^{\frac{1+\alpha}{p}-1} |\varphi(t)| dt + \int_{1/\delta}^{\infty} t^{\frac{1+\alpha}{p}-1} |\varphi(t)| dt\right]  \to 0
	\end{aligned}$$
	as $\delta\to 0$. Hence, $H(\H_{\varphi} f)= \H_{\varphi}(H(f))$ for all $f\in L^q(\R)\cap L^p_{|\cdot|^\alpha}(\R)$. This, together with the density of $L^q(\R)\cap L^p_{|\cdot|^\alpha}(\R)$ in $L^p_{|\cdot|^\alpha}(\R)$, allows us to conclude that 
	$$H(\H_{\varphi} f)= \H_{\varphi}(H(f))$$ 
	for all $f\in L^p_{|\cdot|^\alpha}(\R)$. This ends the proof of Theorem \ref{16:42, 26/4/2025}.
	
\end{proof}

\begin{acknowledgement}
	The authors would like to thank the referees for their careful reading and valuable suggestions, which have greatly improved the quality of the paper. This paper was completed during their visit to the Vietnam Institute for Advanced Study in Mathematics (VIASM), whose financial support and hospitality are gratefully acknowledged. This research is funded by Vietnam National Foundation for Science and Technology Development (NAFOSTED) under grant number 101.02-2023.09.
\end{acknowledgement}

%
%

\end{document}